\documentclass{amsart}

\usepackage{graphicx}
\usepackage{amsmath,amssymb,amsthm}
\usepackage{enumerate}

\newtheorem{theorem}{Theorem}[section]

\newtheorem{proposition}[theorem]{Proposition}

\theoremstyle{definition}

\newtheorem*{main}{Main Theorem}

\theoremstyle{remark}

\numberwithin{equation}{section}

\begin{document}

\title[essential spectra]{On the essential spectra of Toeplitz operators on Bergman spaces of the bi-disc}

\author{U\v{g}ur G\"{u}l}
\address{Hacettepe University, Department of Mathematics, 06800, Beytepe, Ankara, TURKEY}
\email{gulugur@gmail.com}






\begin{abstract}
In this note we deal with the problem of determining the essential
spectrum of a Toeplitz operator
$T_{f}:A^{2}(\mathbb{D}^{2})\rightarrow A^{2}(\mathbb{D}^{2})$
acting on the Bergman space $A^{2}(\mathbb{D}^{2})$ of the bi-disc
whose symbol $f\in C(\overline{\mathbb{D}^{2}})$ lies in the space
of continuous functions on the closure $\overline{\mathbb{D}^{2}}$
of the bi-disc.
\end{abstract}

\maketitle

\section{Introduction}

In this small note we relate the essential spectrum of a Toeplitz
operator $T_{f}:A^{2}(\mathbb{D}^{2})\rightarrow
A^{2}(\mathbb{D}^{2})$ with symbol $f\in
C(\overline{\mathbb{D}^{2}})$ to the spectra of the family of
Toeplitz operators
$T_{f_{\theta_{1}}}:A^{2}(\mathbb{D})\rightarrow
A^{2}(\mathbb{D})$ and
$T_{f_{\theta_{2}}}:A^{2}(\mathbb{D})\rightarrow
A^{2}(\mathbb{D})$ where $f_{\theta_{1}}(w):=f(e^{i\theta_{1}},w)$
and $f_{\theta_{2}}(z):=f(z,e^{i\theta_{2}})$.

We use a tensor product algebra technique to characterize the
C*-algebra\\
$\mathcal{T}(C(\overline{\mathbb{D}^{2}}))/K(A^{2}(\mathbb{D}^{2})$
as a tensor product C*-algebra\\
$\mathcal{T}(C(\overline{\mathbb{D}}))\otimes\mathcal{T}(C(\overline{\mathbb{D}}))/K(A^{2}(\mathbb{D}))\otimes
K(A^{2}(\mathbb{D}))$ and to embed this tensor product C*-algebra
inside the direct sum
$C(\mathbb{T},\mathcal{T}(C(\overline{\mathbb{D}})))\oplus
C(\mathbb{T},\mathcal{T}(C(\overline{\mathbb{D}})))$. We then
apply Atkinson's theorem to characterize the essential spectrum
$\sigma_{e}(T_{f})$ of the Toeplitz operator acting on the Bergman
space $A^{2}(\mathbb{D}^{2})$ of the bi-disc with continuous
symbol $f\in C(\overline{\mathbb{D}^{2}})$.

Actually our method is a modification of the same method in
\cite{DouglasHowe} which was applied to the same problem with
Hardy spaces. In particular we prove the following theorem:
\begin{main}
Let $f\in C(\overline{\mathbb{D}^{2}})$,
$T_{f}:A^{2}(\mathbb{D}^{2})\rightarrow A^{2}(\mathbb{D}^{2})$ be
the Toeplitz operator with symbol $f\in
C(\overline{\mathbb{D}^{2}})$ and $\sigma_{e}(T_{f})$ be the
essential spectrum of $T_{f}$. Then we have
$$\sigma_{e}(T_{f})=\bigg(\bigcup_{e^{i\theta_{1}}\in\mathbb{T}}\sigma(T_{f_{\theta_{1}}})\bigg)\cup\bigg(\bigcup_{e^{i\theta_{2}}\in\mathbb{T}}\sigma(T_{f_{\theta_{2}}})\bigg)$$
where $f_{\theta_{1}}(w):=f(e^{i\theta_{1}},w)$,
$f_{\theta_{2}}(z):=f(z,e^{i\theta_{2}})$, $T_{f_{\theta_{1}}}$
and $T_{f_{\theta_{2}}}$ are associated Toeplitz operators on
$A^{2}(\mathbb{D})$ with symbols $f_{\theta_{1}}$ and
$f_{\theta_{2}}$ respectively.
\end{main}

\section{Preliminaries}
In this section we give basic definitions and preliminary material
that we will use throughout.

Throughout the paper $\mathbb{D}:=\{z\in\mathbb{C}:\mid z\mid<1\}$
will denote the open unit disc, $\mathbb{T}:=\{z\in\mathbb{C}:\mid
z\mid=1\}$ will denote the unit circle. The Bergman space
$A^{2}(\mathbb{D})$ is defined to be the space of all analytic
functions $f:\mathbb{D}\rightarrow\mathbb{C}$ satisfying
$$\iint\limits_{\mathbb{D}}\mid f(z)\mid^{2}dA(z)<+\infty$$
The Bergman space $A^{2}(\mathbb{D})$ is a Hilbert space with
inner product
$$\langle f,g\rangle:=\iint\limits_{\mathbb{D}}f(z)\overline{g(z)}dA(z)$$
The Bergman space $A^{2}(\mathbb{D}^{2})$ of the bi-disc is
defined to be the space of all analytic functions
$f:\mathbb{D}^{2}\rightarrow\mathbb{C}$ satisfying
$$\iiiint\limits_{\mathbb{D}^{2}}\mid f(z,w)\mid^{2}dA(z)dA(w)<+\infty$$
As in one variable case the Bergman space $A^{2}(\mathbb{D}^{2})$
is also a Hilbert space with inner product
$$\langle f,g\rangle:=\iiiint\limits_{\mathbb{D}^{2}}f(z,w)\overline{g(z,w)}dA(z)dA(w)$$
The space of all continuous function on $\overline{\mathbb{D}}$ is
denoted by $C(\overline{\mathbb{D}})$, likewise $C(\mathbb{T})$
denotes the space of all continuous functions on $\mathbb{T}$ and
$C(\overline{\mathbb{D}^{2}})$ denotes the space of all continuous
functions on $\overline{\mathbb{D}^{2}}$.

For $f\in C(\overline{\mathbb{D}^{2}})$, the Toeplitz operator
$T_{f}:A^{2}(\mathbb{D}^{2})\rightarrow A^{2}(\mathbb{D}^{2})$ is
defined to be the operator $T_{f}:=P_{2}M_{f}$ where
$P_{2}:L^{2}(\mathbb{D}^{2})\rightarrow A^{2}(\mathbb{D}^{2})$ is
the unique orthogonal projection of $L^{2}(\mathbb{D}^{2})$ onto
$A^{2}(\mathbb{D}^{2})$ and $(M_{f}g)(z,w):=f(z,w)g(z,w)$ is the
multiplication operator. Likewise for $\varphi\in
C(\overline{\mathbb{D}})$ the Toeplitz operator
$T_{\varphi}:A^{2}(\mathbb{D})\rightarrow A^{2}(\mathbb{D})$ is
defined to be the operator $T_{\varphi}:=PM_{\varphi}$ where
$P:L^{2}(\mathbb{D})\rightarrow A^{2}(\mathbb{D})$ is the
orthogonal projection of $L^{2}(\mathbb{D})$ onto
$A^{2}(\mathbb{D})$ and $(M_{\varphi}g)(z):=\varphi(z)g(z)$ is the
multiplication operator.

Recall that a bounded linear operator $T$ on a Hilbert space $H$
is called Fredholm if the range of $T$ is closed, $\dim\ker(T)$
and $\dim\ker(T^{\ast})$ are finite. The Fredholm index $ind$ is
defined as
$$ind(T)=\dim(\ker(T))-\dim(\ker(T^{\ast}))$$
The following Atkinson's characterization for Fredholm operators
is also well known:
\begin{theorem}{\cite[p.28, Theorem 1.4.16]{Murphy}}\label{Atkinson}
A bounded linear operator $T$ on a Hilbert space $H$ is Fredholm
if and only if
 $T+K(H)$ is invertible in the quotient algebra $B(H)/K(H)$, where $K(H)$ is the algebra of all compact operators on $H$.
\end{theorem}

Let $H_{1}$ and $H_{2}$ be two given Hilbert spaces. On the
algebraic tensor product $H_{1}\otimes H_{2}$ of $H_{1}$ and
$H_{2}$, there is a unique inner product $\langle.,.\rangle$
satisfying the following equation
\begin{equation*}
\langle x_{1}\otimes y_{1},x_{2}\otimes y_{2}\rangle=\langle
x_{1},x_{2}\rangle_{H_{1}}\langle y_{1},y_{2}\rangle_{H_{2}}
\end{equation*}
$\forall x_{1},x_{2}\in H_{1}\quad y_{1},y_{2}\in H_{2}$ (See
\cite[p.185]{Murphy}). For any $T\in B(H_{1})$ and $S\in B(H_{2})$
there is a unique operator $T\hat{\otimes} S\in$ $B(H_{1}\otimes
H_{2})$ satisfying the following equation:
\begin{equation*}
(T\hat{\otimes} S)(x\otimes y)=Tx\otimes Sy
\end{equation*}
For any two C*-algebras $A\subset B(H_{1})$ and $B\subset
B(H_{2})$ the algebraic tensor product $A\odot B$ is defined to be
the linear span of operators of the form $T\hat{\otimes} S$ i.e.
\begin{equation*}
A\odot B=\{\sum_{j=1}^{n}T_{j}\hat{\otimes} S_{j}:T_{j}\in A,\quad
S_{j}\in B\}
\end{equation*}
The algebraic tensor product $A\odot B$ becomes a *-algebra with
multiplication
\begin{equation*}
(T_{1}\hat{\otimes} S_{1})(T_{2}\hat{\otimes} S_{2})=
T_{1}T_{2}\hat{\otimes} S_{1}S_{2}
\end{equation*}
and involution
\begin{equation*}
(T\hat{\otimes} S)^{*}=T^{*}\hat{\otimes} S^{*}
\end{equation*}
However there might be more than one norm making the closure of
$A\odot B$ into a $C^*$-algebra. A $C^*$-algebra $A$ is called
``nuclear" if for any $C^*$-algebra $B$ there is a unique pre
$C^*$-algebra norm on the algebraic tensor product $A\odot B$ of
$A$ and $B$. A well-known theorem of Takesaki asserts that any
commutative C*-algebra is nuclear (see \cite[p.205]{Murphy}). An
extension of a C*-algebra by nuclear C*-algebras is nuclear, i.e.
if $A$, $B$ and $C$ are C*-algebras s.t. the following sequence
\begin{equation*}
0\xrightarrow{} A\xrightarrow{j} B\xrightarrow{\pi}
C\xrightarrow{} 0
\end{equation*}
is short exact and $A$ and $C$ are nuclear then $B$ is also
nuclear (see \cite[p. 212]{Murphy}). For any separable Hilbert
space $H$ the C*-algebra of all compact operators $K(H)$ on $H$ is
nuclear (see \cite{Murphy} p. 183 and 196). For any separable
Hilbert spaces $H_{1}$ and $H_{2}$ we have
\begin{equation}
K(H_{1}\otimes H_{2})=K(H_{1})\otimes K(H_{2})
\end{equation}
(See \cite[p.207]{DouglasHowe}).

The essential spectrum $\sigma_{e}(T)$ of an operator $T$ acting
on a Banach
  space $X$ is the spectrum of the coset of $T$ in the Calkin algebra
  $B(X)/K(X)$, the algebra of bounded linear operators modulo
  compact operators. The well known Atkinson's theorem identifies the essential
  spectrum of $T$ as the set of all $\lambda\in$ $\mathbb{C}$ for
  which $\lambda I-T$ is not a Fredholm operator.

To give a Fredholm criteria of the operators in
$\mathcal{T}(C(\overline{\mathbb{D}^{2}}))/K(A^{2}(\mathbb{D}^{2})$,
we will use Atkinson's characterization. Hence, firstly we need to
know the structure of the quotient
$\mathcal{T}(C(\overline{\mathbb{D}^{2}}))/K(A^{2}(\mathbb{D}^{2})$.
To investigate this structure we use Douglas and Howe's method in
\cite{DouglasHowe} and identify the space $A^2(\mathbb{D}^2)$ as
the tensor product of one variable space $A^{2}(\mathbb{D})$ with
itself. For this purpose, we recall the one-variable case studied
by Coburn \cite{Coburn}. In \cite{Coburn} Coburn showed that the
following sequence
\begin{equation}\label{shortexactseq}
0\longrightarrow K(A^2(\mathbb{D}))\stackrel{j}{\longrightarrow}
\mathcal{T}(C(\overline{\mathbb{D}}))\stackrel{\pi}{\longrightarrow}
C(\mathbb{T})\longrightarrow 0
\end{equation}
is short exact, where
\[\mathcal{T}(C(\mathbb{D}))=C^*(\{T_{\varphi}:\varphi\in C(\overline{\mathbb{D}})\})\]
is the $C^*$-algebra generated by Toeplitz operators with
continuous symbols and $\pi$ is the restriction map
$\pi(T_{\varphi}):=\varphi|_{\mathbb{T}}$.

Throughout the $C^*$-algebra
$\mathcal{T}(C(\overline{\mathbb{D}}))$ will be denoted by
$\mathcal{T}$ and similarly,
$\mathcal{T}(C(\overline{\mathbb{D}^{2}}))$ will be denoted by
$\mathcal{T}_2$. Since $K(A^2(\mathbb{D}))$ is nuclear and
$C(\mathbb{T})$ is commutative and hence nuclear, by Equation
\eqref{shortexactseq} $\mathcal{T}$ is nuclear. Therefore all the
$C^*$-algebras that we will deal with in this paper will be
nuclear. Since the C*-algebras are nuclear it is possible to
identify $\mathcal{T}_2$ with $\mathcal{T}\otimes\mathcal{T}$
corresponding to the identification of $A^2(\mathbb{D}^2)$ with
$A^2(\mathbb{D})\otimes A^2(\mathbb{D})$. Moreover by equation
(2.1) we know that
\[K(A^2(\mathbb{D}^2))=K(A^2(\mathbb{D}))\otimes K(A^2(\mathbb{D}))\] and that $K(A^{2}(\mathbb{D}^{2}))$ is properly contained in the
commutator ideal of $\mathcal{T}_2$.

We now recall the following propositions. Firstly,

\begin{theorem}{\cite[Theorem 6.5.2, p.211]{Murphy}}\label{tensor}
Let $\mathcal{J}$, $\mathcal{A}$, $\mathcal{B}$ and $\mathcal{D}$
be $C^*$-algebras and suppose that
\[0\longrightarrow \mathcal{J}\stackrel{j}{\longrightarrow}\mathcal{A}\stackrel{\pi}{\longrightarrow} \mathcal{B}\longrightarrow 0 \]
is a short exact sequence of $C^*$-algebras. Suppose also that
$\mathcal{B}\otimes\mathcal{D}$ has a unique $C^*$-norm (this is
the case if $\mathcal{B}$ or $\mathcal{D}$ is nuclear). Then
\[0\longrightarrow \mathcal{J}\otimes \mathcal{D}\stackrel{j\otimes 1}{\longrightarrow} \mathcal{A}\otimes\mathcal{D} \stackrel{\pi\otimes 1}{\longrightarrow} \mathcal{B}\otimes \mathcal{D}\longrightarrow 0 \]
is a short exact sequence of $C^*$-algebras.
\end{theorem}
Second, we need the following standard result concerning tensor
products:
\begin{proposition}{\cite[Proposition 3]{DouglasHowe}}
If $X$ is a compact Hausdorff space and $\mathcal{B}$ is a
$C^*$-algebra, then $C(X)\otimes\mathcal{B}$ is naturally
isomorphic to the $C^*$-algebra $C(X,\mathcal{B})$ of continuous
functions from $X$ to $\mathcal{B}$. Moreover, if $\sum_{i=1}^n
f_i\otimes B_i$ in $C(X)\otimes\mathcal{B}$ corresponds to $F$ in
$C(X,\mathcal{B})$, then $F(x)=\sum_{i=1}^n f_i(x)B_i$.
\end{proposition}

\section{Main Results}

In this section we prove the main result of this paper which gives
a criteria of Fredholmness of the operators in $\mathcal{T}_2$. In
doing this we use the Atkinson's characterization for the Fredholm
operators. Hence, first of all, we need to know the structure of
the $C^*$-algebra $\mathcal{T}_2/K(A^2(\mathbb{D}^2))$. For this,
applying Theorem \eqref{tensor} to the short exact sequence
\eqref{shortexactseq}, we set the following commutative diagram
with exact rows and columns.

$\begin{array}[c]{lccccccll}
   &  & 0\hspace*{0.5cm} &  & 0\hspace*{0.5cm} &  & 0\hspace*{0.5cm} & &  \\
   &  &  \downarrow{\hspace*{0.5cm}} &  &  \downarrow{\hspace*{0.5cm}} &  &  \downarrow{\hspace*{0.5cm}} &  &  \\
  0 & \longrightarrow & K(A^2(\mathbb{D}))\otimes K(A^2(\mathbb{D})) & \stackrel{j\otimes1}{\longrightarrow} &\mathcal{T}\otimes K(A^2(\mathbb{D})) &\stackrel{\pi\otimes1}{\longrightarrow} & C(\mathbb{T})\otimes K(A^2(\mathbb{D})) & \longrightarrow & 0 \\
   &  &  \downarrow{\vspace*{4cm}}\scriptstyle{1\otimes j} &  &  \downarrow{\vspace*{1cm}}\scriptstyle{1\otimes j} &  &  \downarrow{\vspace*{1cm}}\scriptstyle{1\otimes j} &  &  \\
  0 & \longrightarrow & K(A^2(\mathbb{D}))\otimes\mathcal{T} & \stackrel{j\otimes1}{\longrightarrow} & \mathcal{T}\otimes\mathcal{T} &\stackrel{\pi\otimes1}{\longrightarrow} & C(\mathbb{T})\otimes\mathcal{T} & \longrightarrow & 0 \\
   &  &  \downarrow{\vspace*{2cm}}\scriptstyle{1\otimes \pi} &  &  \downarrow{\vspace*{2cm}}\scriptstyle{1\otimes \pi} &  &  \downarrow{\vspace*{2cm}}\scriptstyle{1\otimes \pi} &  &  \\
  0 & \longrightarrow & K(A^2(\mathbb{D}))\otimes C(\mathbb{T}) & \stackrel{j\otimes1}{\longrightarrow} & \mathcal{T}\otimes C(\mathbb{T}) &\stackrel{\pi\otimes1}{\longrightarrow} & C(\mathbb{T})\otimes C(\mathbb{T}) & \longrightarrow & 0 \\
   &  &  \downarrow{\hspace*{0.5cm}} &  &  \downarrow{\hspace*{0.5cm}} &  &  \downarrow{\hspace*{0.5cm}} &  &  \\
   &  & 0\hspace*{0.5cm} &  & 0\hspace*{0.5cm} &  & 0\hspace*{0.5cm} & &  \\
\end{array}$
Consideration of the diagonal map shows that the commutator ideal
of $\mathcal{T}\otimes\mathcal{T}$ is the subspace spanned by
$\mathcal{T}\otimes K(A^2(\mathbb{D}))$ and
$K(A^2(\mathbb{D}))\otimes\mathcal{T}$ and the corresponding
quotient is $C(\mathbb{T})\otimes C(\mathbb{T})$ which is
naturally isomorphic to $C(\mathbb{T}\times\mathbb{T})$. We are
primarily interested, however, in the dual sequence
\[ ...\longrightarrow K(A^2(\mathbb{D}))\otimes K(A^2(\mathbb{D})) \stackrel{\alpha}{\longrightarrow} \mathcal{T}\otimes\mathcal{T} \stackrel{(\pi\otimes1)\oplus(1\otimes\pi)}{\longrightarrow}  (C(\mathbb{T})\otimes\mathcal{T})\oplus (\mathcal{T}\otimes C(\mathbb{T}))  \longrightarrow ...  \]
where $\alpha=(1\otimes j)(j\otimes 1)=(j\otimes 1)(1\otimes j)$
which is exact in the middle. The proof involves routine diagram chasing.\\

We now can prove the following result.

\begin{proposition}\label{restate}
There exist $*$-homomorphisms $\gamma_{\theta_1}$ and
$\gamma_{\theta_2}$ from $\mathcal{T}_2$ into
$C(\mathbb{T},\mathcal{T})$ such that $\gamma_{\theta_1}(T_f)=F$
and $\gamma_{\theta_2}(T_f)=G$, where
$F(e^{i\theta_{1}})=T_{f(e^{i\theta_{1}},\cdot)}$ and
$G(e^{i\theta_{2}})=T_{f(\cdot,e^{i\theta_{2}})}$ and $f\in
C(\overline{\mathbb{D}^{2}})$. Moreover, under natural
identification of $\mathcal{T}_2$ with
$\mathcal{T}\otimes\mathcal{T}$, $C(\mathbb{T})\otimes\mathcal{T}$
with $C(\mathbb{T},\mathcal{T})$ and $\mathcal{T}\otimes
C(\mathbb{T})$ with $C(\mathbb{T},\mathcal{T})$, then
$\pi\otimes1=\gamma_{\theta_1}$ and
$1\otimes\pi=\gamma_{\theta_2}$.
\end{proposition}

\begin{proof}
Since $\pi\otimes1$ is a $*$-homomorphism from
$\mathcal{T}\otimes\mathcal{T}$ into
$C(\mathbb{M})\otimes\mathcal{T}$, it induces a $*$-homomorphism
$\gamma_{\theta_1}$ from $\mathcal{T}_2$ onto
$C(\mathbb{T},\mathcal{T})$ and we need only compute the action of
$\gamma_{\theta_1}$ on $T_f$ for $f\in
C(\overline{\mathbb{D}^{2}})$. If
$\phi_1,\phi_2,\ldots,\phi_{2n}\in C(\overline{\mathbb{D}})$, then
$f(z_1,z_2)=\sum_{i=1}^n \phi_i(z_1)\phi_{n+i}(z_2)$ is in
$C(\overline{\mathbb{D}})\otimes C(\overline{\mathbb{D}})$.
Moreover, since $T_f$ corresponds to $(\sum_{i=1}^n
T_{\phi_i}\otimes T_{\phi_{n+i}})$, we have $(\pi\otimes
1)[(\sum_{i=1}^n T_{\phi_i}\otimes T_{\phi_{n+i}})]=(\sum_{i=1}^n
\phi_i\otimes T_{\phi_{n+i}})$ and if we set
$\gamma_{\theta_1}(T_f)=F$, then $F(e^{i\theta_1})=(\sum_{i=1}^n
\phi_i(e^{i\theta_1}) T_{\phi_{n+i}})$. Lastly, since the
functions of the form $f$ are dense in
$C(\overline{\mathbb{D}^{2}})$, the result follows.
\end{proof}

We now restate the preceding exactness result as follows:

\begin{theorem}\label{quotientcompact}
The sequence
\begin{equation}
...\longrightarrow
K(A^2(\mathbb{D}^2))\stackrel{i}{\longrightarrow}
\mathcal{T}_2\stackrel{\gamma_{\theta_1}\oplus
\gamma_{\theta_2}}{\longrightarrow}
C(\mathbb{T},\mathcal{T})\oplus
C(\mathbb{T},\mathcal{T})\longrightarrow ...
\end{equation}
is exact in the middle.
\end{theorem}
In other words, the quotient algebra
$\mathcal{T}_2/K(A^2(\mathbb{D}^2))$ is isometrically $*$-embedded
in $C(\mathbb{T},\mathcal{T})\oplus C(\mathbb{T},\mathcal{T})$.

In view of the Atkinson's characterization, Fredholmness for an
operator in $\mathcal{T}_{2}$ is closely related to invertibility
criteria for operators in $\mathcal{T}$ by the theorem above. So
we have our main result as follows:

\begin{main}
Let $f\in C(\overline{\mathbb{D}^{2}})$,
$T_{f}:A^{2}(\mathbb{D}^{2})\rightarrow A^{2}(\mathbb{D}^{2})$ be
the Toeplitz operator with symbol $f\in
C(\overline{\mathbb{D}^{2}})$ and $\sigma_{e}(T_{f})$ be the
essential spectrum of $T_{f}$. Then we have
$$\sigma_{e}(T_{f})=\bigg(\bigcup_{e^{i\theta_{1}}\in\mathbb{T}}\sigma(T_{f_{\theta_{1}}})\bigg)\cup\bigg(\bigcup_{e^{i\theta_{2}}\in\mathbb{T}}\sigma(T_{f_{\theta_{2}}})\bigg)$$
where $f_{\theta_{1}}(w):=f(e^{i\theta_{1}},w)$,
$f_{\theta_{2}}(z):=f(z,e^{i\theta_{2}})$, $T_{f_{\theta_{1}}}$
and $T_{f_{\theta_{2}}}$ are associated Toeplitz operators on
$A^{2}(\mathbb{D})$ with symbols $f_{\theta_{1}}$ and
$f_{\theta_{2}}$ respectively.
\end{main}

\begin{proof}
Let $f\in C(\overline{\mathbb{D}^{2}})$, and let
$\lambda\in\mathbb{C}$ then if $\lambda\in\sigma_{e}(T_{f})$ then
by Atkinson's theorem $\lambda-[T_{f}]$ is not invertible in
$B(A^{2}(\mathbb{D}^{2}))/K(A^{2}(\mathbb{D}^{2}))$. Since
$\lambda-[T_{f}]\in\mathcal{T}_{2}/K(A^{2}(\mathbb{D}^{2}))$ and
$\mathcal{T}_{2}/K(A^{2}(\mathbb{D}^{2}))\subset
B(A^{2}(\mathbb{D}^{2}))/K(A^{2}(\mathbb{D}^{2}))$ is a closed
unital C*-subalgebra, it is inverse closed and hence
$\lambda-[T_{f}]\in\mathcal{T}_{2}/K(A^{2}(\mathbb{D}^{2}))$ is
not invertible in $\mathcal{T}_{2}/K(A^{2}(\mathbb{D}^{2}))$
neither. Since $\mathcal{T}_{2}$ is identified with
$\mathcal{T}\otimes\mathcal{T}$ and
$\gamma_{\theta_1}\oplus\gamma_{\theta_2}:(\mathcal{T}\otimes\mathcal{T})/K(A^{2}(\mathbb{D}^{2}))\rightarrow
C(\mathbb{T},\mathcal{T})\oplus C(\mathbb{T},\mathcal{T})$ is a
*-homomorphism s.t.
$\gamma_{\theta_1}\oplus\gamma_{\theta_2}(T_{f}):=F\oplus G$,
$F(e^{i\theta_1}):=T_{f_{\theta_1}}$,
$f_{\theta_1}(w):=f(e^{i\theta_1},w)$ and
$G(e^{i\theta_2}):=T_{f_{\theta_2}}$,
$f_{\theta_2}(z):=f(z,e^{i\theta_2})$, $\lambda-F$ or $\lambda-G$
is not invertible in $C(\mathbb{T},\mathcal{T})$. This is the case
iff there is $e^{i\theta_1}\in\mathbb{T}$ or
$e^{i\theta_2}\in\mathbb{T}$ s.t. either
$\lambda-T_{f_{\theta_1}}$ or $\lambda-T_{f_{\theta_2}}$ is not
invertible in $\mathcal{T}$. And since $\mathcal{T}\subset
B(A^{2}(\mathbb{D}))$ is a closed C*-subalgebra of
$B(A^{2}(\mathbb{D}))$, it is inverse closed hence either
$\lambda-T_{f_{\theta_1}}$ or $\lambda-T_{f_{\theta_2}}$ is not
invertible in $B(A^{2}(\mathbb{D}))$. And this implies that
$\lambda\in\sigma(T_{f_{\theta_1}})$ or
$\lambda\in\sigma(T_{f_{\theta_2}})$. So we have
$$\sigma_{e}(T_{f})\subseteq\bigg(\bigcup_{e^{i\theta_{1}}\in\mathbb{T}}\sigma(T_{f_{\theta_{1}}})\bigg)\cup\bigg(\bigcup_{e^{i\theta_{2}}\in\mathbb{T}}\sigma(T_{f_{\theta_{2}}})\bigg)$$
For the reverse inclusion let
$\lambda\in\bigg(\bigcup_{e^{i\theta_{1}}\in\mathbb{T}}\sigma(T_{f_{\theta_{1}}})\bigg)\cup\bigg(\bigcup_{e^{i\theta_{2}}\in\mathbb{T}}\sigma(T_{f_{\theta_{2}}})\bigg)$,
then there is $e^{i\theta_1}\in\mathbb{T}$ or
$e^{i\theta_2}\in\mathbb{T}$ s.t. $\lambda-T_{f_{\theta_1}}$ or
$\lambda-T_{f_{\theta_2}}$ is not invertible in
$B(A^{2}(\mathbb{D}))$. Again since
$\lambda-T_{f_{\theta_1}}\in\mathcal{T}$,
$\lambda-T_{f_{\theta_2}}\in\mathcal{T}$ and $\mathcal{T}$ is
inverse closed in $B(A^{2}(\mathbb{D}))$,
$\lambda-T_{f_{\theta_1}}$ or $\lambda-T_{f_{\theta_2}}$ is not
invertible in $\mathcal{T}$. And this implies that $\lambda-F$ or
$\lambda-G$ is not invertible in $C(\mathbb{T},\mathcal{T})$.
Since $\gamma_{\theta_1}\oplus\gamma_{\theta_2}$ is a *-embedding
of
$\mathcal{T}\otimes\mathcal{T}/K(A^{2}(\mathbb{D}^{2}))\cong\mathcal{T}_{2}/K(A^{2}(\mathbb{D}^{2}))$
inside $C(\mathbb{T},\mathcal{T})\oplus
C(\mathbb{T},\mathcal{T})$, this implies that $\lambda-T_{f}$ is
not invertible in $\mathcal{T}_{2}/K(A^{2}(\mathbb{D}^{2}))$.
Again the inverse closedness of
$\mathcal{T}_{2}/K(A^{2}(\mathbb{D}^{2}))$ in
$B(A^{2}(\mathbb{D}^{2}))/K(A^{2}(\mathbb{D}^{2}))$ implies that
$\lambda-T_{f}$ is not invertible in
$B(A^{2}(\mathbb{D}^{2}))/K(A^{2}(\mathbb{D}^{2}))$ neither. Hence
again by Atkinson's theorem we have $\lambda\in\sigma_{e}(T_{f})$.
So we also have the reverse inclusion
$$\bigg(\bigcup_{e^{i\theta_{1}}\in\mathbb{T}}\sigma(T_{f_{\theta_{1}}})\bigg)\cup\bigg(\bigcup_{e^{i\theta_{2}}\in\mathbb{T}}\sigma(T_{f_{\theta_{2}}})\bigg)\subseteq\sigma_{e}(T_{f})$$
which implies that
$$\sigma_{e}(T_{f})=\bigg(\bigcup_{e^{i\theta_{1}}\in\mathbb{T}}\sigma(T_{f_{\theta_{1}}})\bigg)\cup\bigg(\bigcup_{e^{i\theta_{2}}\in\mathbb{T}}\sigma(T_{f_{\theta_{2}}})\bigg)$$
\end{proof}

\begin{center}
\textbf{Acknowledgements}
\end{center}
The author would like to express his sincere thanks to his dear
friend Prof. S\"{o}nmez \c{S}ahuto\u{g}lu of Toledo University for
the main question of this paper.


\begin{thebibliography}{20}

\bibitem{Coburn}
        {\sc L. A. Coburn},
        {\it Singular Integral Operators and Toeplitz Operators on Odd Spheres},
        Indiana University Mathematics Journal, Vol. 23, No. 5 (1973), pp. 433-439

\bibitem{Douglas}
        {\sc R.G. Douglas},
        {\it Banach algebra techniques in Operator Theory},
        GTM Springer, 1998.

\bibitem{DouglasHowe}
        {\sc R.G. Douglas, R. Howe},
        {\it On the $C^*$-algebra of Toeplitz operators on the quarterplane},
        Trans. Amer. Math. Soc. \bf{158} (1971) 203--217.


\bibitem{Murphy}
        {\sc G.J. Murphy},
        {\it $C^*$-algebras and Operator Theory},
        Academic Press, New York-London, 1990.


\end{thebibliography}
\end{document}